\documentclass[12pt]{amsart}
\usepackage{fullpage,amssymb}

\usepackage{mathrsfs}

\newtheorem{theorem}{Theorem}[section]
\newtheorem{lemma}[theorem]{Lemma}
\newtheorem{corollary}[theorem]{Corollary}
\newtheorem{proposition}[theorem]{Proposition}
\newtheorem*{notation*}{Notation}
\newtheorem*{p*}{Proposition~\ref{h.s.o.p}}
\theoremstyle{definition}
\newtheorem{definition}[theorem]{Definition}
\newtheorem{example}[theorem]{Example}
\newtheorem{remark}[theorem]{Remark}
\newcommand{\M}{{\operatorname{Mat}}}

\newcommand{\Z}{{\mathbb Z}}

\newcommand{\SL}{{\operatorname{SL}}}

\newcommand{\spa}{\operatorname{span}}

\newcommand{\Q}{{\mathbb Q}}

\newcommand{\rk}{\operatorname{rk}}

\usepackage{multirow}
\usepackage{amsmath}
\usepackage{comment}
\usepackage{tikz}
\usepackage{mathtools}
\usepackage{arydshln}
\usepackage{graphicx}

\title{The regularity lemma is false over small fields}
\author{Harm Derksen}
\address{Department of Mathematics, University of Michigan.}
\email{hderksen@umich.edu}

\author{Visu Makam}
\address{School of Mathematics, Institute for Advanced Study, Princeton.}
\email{visu@ias.edu}
\thanks{This material is based upon work supported by the National Science Foundation under Grant No. DMS-
1601229 and DMS-1638352.}


\begin{document}
\maketitle

\begin{abstract}
The regularity lemma is a stringent condition of the possible ranks of tensor blow-ups of linear subspaces of matrices. It was proved by Ivanyos, Qiao and Subrahmanyam in \cite{IQS} when the underlying field is sufficiently large. We show that if the field size is too small, the regularity lemma is false.
\end{abstract}

\section{Introduction}
\subsection{Motivation}
A compelling reason to study linear subspaces of matrices is its connection to the problem of Polynomial Identity Testing (PIT), see \cite{Valiant,GGOW}. A polynomial time algorithm for PIT would be a major breakthrough in the field of computational complexity. While an efficient algorithm for PIT remains elusive, recently efficient algorithms for the related problem of Rational Identity Testing (RIT) have been found. An analytic algorithm was given in \cite{GGOW} when the underlying field is $\Q$. For sufficiently large fields, an algebraic algorithm was given in \cite{IQS,IQS2} based on polynomial degree bounds for matrix semi-invariants proved in \cite{DM}. An excellent introduction to these problems, their connection to various subjects, and their significance can be found in \cite{GGOW}.

Crucial to the algebraic algorithm is the regularity lemma proved in \cite{IQS}. Curiously enough, the proof of the regularity lemma requires a sufficiently large field. The authors of \cite{IQS} pose the open problem of whether the regularity lemma holds over small fields as well, see \cite[Remark~5.9]{IQS}. The purpose of this note is to give counterexamples to the regularity lemma when the field is too small. The counterexamples can be generated naturally from understanding the alternate proof of the regularity lemma given in \cite{DM-ncrk}.

\subsection{Definitions and main results}
Let $K$ be a field. Let $\M_{p,q}$ denote the space of $p \times q$ matrices with entries in $K$, and let $\mathcal{X} \subset \M_{p,q}$ be a linear subspace of $p \times q$ matrices.
\begin{definition}
We define the rank of $\mathcal{X}$ to be 
$$
\rk(\mathcal{X}) = \max\{\rk(X)\ |\ X \in \mathcal{X}\}.
$$
\end{definition}  

Given two matrices $A \in \M_{p,q}$ and $B \in \M_{r,s}$, we recall that their Kronecker product is

$$A \otimes B = \begin{bmatrix}
a_{11}B   & a_{12}B     &  \cdots     & a_{1n}B  \\
 a_{21}B  &   \ddots               &                 & \vdots  \\
 \vdots      &       &    \ddots  &\vdots   \\
 a_{m1}B  &       \cdots       & \cdots               & a_{mn}B \\
 \end{bmatrix}\in \M_{pr,qs}.$$

The notion of tensor blow-ups of linear subspaces was introduced by Ivanyos, Qiao and Subrahmanyam in \cite{IQS}.

\begin{definition}
We define the $(r,s)^{th}$ blow-up of the linear subspace of matrices $\mathcal{X}$ to be 
$$
\mathcal{X}^{\{r,s\}} := \mathcal{X} \otimes \M_{r,s} = \{\sum_i X_i \otimes M_i\ |\ X_i \in \mathcal{X}, M_i \in \M_{r,s} \} \subseteq \M_{pr,qs}.
$$
\end{definition}

Tensor blow-ups arise naturally in the study of the left-right action of $\SL_n \times \SL_n$ on $\M_{n,n}^m$. Given a tuple $X = (X_1,\dots,X_m) \in \M_{n,n}^m$, let $\mathcal{X} = \spa (X_1,\dots,X_m) \subseteq \M_{n,n}$. Then it turns out that $X$ is in the null cone for the left-right action if and only if $\mathcal{X}^{\{d,d\}}$ has full rank for some $d$. We refer the reader to \cite{DM,IQS} for more details. This invariant theoretic view point emphasizes the need to understand the properties of ranks of tensor blow-ups.  Two such results are the regularity lemma proved in \cite{IQS}, and the concavity property shown in \cite{DM}. In this paper, we are focused on the regularity lemma.

\begin{theorem} [Regularity lemma, \cite{IQS}] \label{reg.lemma}
Let $\mathcal{X} \subseteq \M_{n,n}$ be a linear subspace of matrices. Let $d \in \Z_{>0}$, and let $|K| \geq (dn)^{\Omega(1)}$. Then $\rk(\mathcal{X}^{\{d,d\}})$ is a multiple of $d$.
\end{theorem}

The proof of the regularity lemma in \cite{IQS} is algorithmic in nature. Given a matrix in the $(d,d)^{th}$ blow-up whose rank is $rd + k$ with $k < d$, the proof consists of constructing efficiently a matrix in the $(d,d)^{th}$ blow-up whose rank is at least $(r+1)d$. The procedure requires the lower bound on field size, so it was unclear whether this lower bound on field size is really necessary for the statement to hold. We show that the lower bound on the field size cannot be entirely removed.

\begin{theorem} \label{main}
Let $d,n \in \Z_{\geq 2}$. Suppose $|K| \leq {\rm log}_d (n - 1)$. Consider the $n \times n$ matrix 
$$ 
D(t_0,t_1) = \left[ \begin{array} {c|ccc}
L(t_0,t_1)   &&& \\
 \hline
   &t_0& &  \\
   && \ddots & \\
   &&& t_0 \\
 \end{array} \right], \text{where } L(t_0,t_1) =  \begin{bmatrix} t_1 & t_0 & && & &\\
  & -t_1 & t_0 & & &&\\
  & & - t_1 & t_0 && &\\
  & & & \ddots & \ddots & &\\
   & & & & -t_1 & t_0 \\
   -t_1 & & & & & -t_1 
 \end{bmatrix}
$$
is square of size $|K|^d$. Consider the $2$-dimensional linear subspace $ \mathcal{X} = \{D(x,y)\ |\ x,y \in K \} \subseteq \M_{n,n}$. Then $\rk(\mathcal{X}^{\{d,d\}})$ is not a multiple of $d$.
\end{theorem}

%Even though it is difficult to find particular examples, we will see that there must be examples for all finite fields.

%\begin{theorem}
%For any finite field $K$, there exists a linear subspace $\mathcal{X}$ and $d \in \Z_{\geq 0}$ such that $\rk(\mathcal{X}^{\{d,d\}})$ is not a multiple of $d$.
%\end{theorem}

%A direct proof of the above result could be obtained by enumerating all the matrices in the subspace, and checking their ranks. We will take a more conceptual approach to arrive at the result.

\section{Linear matrices}
Linear subspaces of matrices can be described using the language of linear matrices.

\begin{definition}
A $p \times q$ linear matrix $L(t_1,t_2,\dots,t_m)$ is a $p \times q$ matrix whose entries are linear expressions in a collection of indeterminates $t_1,t_2,\dots t_m$ with coefficients from $K$. Equivalently, we can write $L(t_1,\dots,t_m) = t_1X_1 + t_2X_2 + \dots + t_mX_m$ for some $X_i \in \M_{p,q}(K)$.
\end{definition}

\begin{example}
An example of a $3 \times 3$ linear matrix is given by  $$
L(t_1,t_2,t_3) := \begin{bmatrix} 0 & t_1 & t_2 \\ -t_1 & 0 & t_3 \\ -t_2 & - t_3 & 0 \end{bmatrix}.
$$
\end{example}
A linear matrix can be thought of as a parametrization of a linear subspace of matrices. Indeed, by specializing the indeterminates in a linear matrix to all possible values of the field $K$, we get a linear subspace of matrices. 
In the above example, the linear matrix $L(t_1,t_2,t_3)$ parametrizes the subspace of $3 \times 3$ matrices consisting of all skew-symmetric matrices.

We can even substitute matrices for the indeterminates. More precisely, if $L(t_1,\dots,t_m)$ is a $p \times q$ linear matrix, and $A_1,\dots,A_m$ are $r \times s$ matrices, then $L(A_1,\dots,A_m)$ denotes the $pr \times qs$ matrix obtained by replacting each entry in $L(t_1,\dots,t_m)$ with the corresponding linear combination of $A_i$. 
For instance, in the above example,
$$
L(A_1,A_2,A_3) = \begin{bmatrix} 0 & A_1 & A_2 \\ -A_1 & 0 & A_3 \\ -A_3 & -A_2 & 0 \end{bmatrix} \in \M_{3r,3s}.
$$

Observe that if $L(t_1,\dots,t_m) = t_1X_1 + \dots + t_mX_m$, then $L(A_1,\dots,A_m) = X_1 \otimes A_1 + \dots + X_m \otimes A_m$.

%Conversely, for any spanning set $X_1,\dots,X_m$ of a linear subspace $\mathcal{X}$, we can associate the linear matrix $L = t_1X_1 + t_2X_2 + \dots t_mX_m$. Clearly, the linear subspace parametrized by $L$ is $\mathcal{X}$.

\begin{definition}
The rank of a linear matrix $L(t_1,\dots,t_m)$ is defined as 
$$
\rk(L(t_1,\dots,t_m)) = \max\{ \rk(L(a_1,\dots,a_m))\ |\ a_i \in K\}.
$$
\end{definition}

\begin{definition}
We define the $(r,s)^{th}$ blow-up rank of a linear matrix $L(t_1,\dots,t_m)$ as 
$$
\rk(L^{\{r,s\}}) = \max\{ \rk(L(A_1,\dots,A_m)\ |\ A_i \in \M_{r,s}(K) \}.
$$
\end{definition}

\begin{remark}
For a linear matrix $L(t_1,\dots,t_m)$, let $\mathcal{X}_L$ denote the associated linear subspace of matrices. Then it is easy to see that $\rk(L) = \rk(\mathcal{X}_L)$ and $\rk(L^{\{r,s\}}) = \rk(\mathcal{X}_L^{\{r,s\}})$.
\end{remark}

\section{Ring of matrix functions} \label{gen.mat}
For this section, let us fix $d \in \Z_{\geq 2}$. Consider the skew polynomial ring $A_{d,m} = K \left< T_1,\dots,T_m \right>$ in $n$ indeterminates. Any $f(T_1,\dots,T_m) \in A_{d,m}$ can be interpreted as a function $ev_d (f) : \M_{d,d}^m$ to $\M_{d,d}$ given by $ev_d(f) (A_1,\dots,A_m) = f(A_1,\dots,A_m)$. This gives a homomorphism
$$
ev_d: A_{d,m} \rightarrow {\rm Hom}_{\rm Sets} (\M_{d,d}^m, \M_{d,d}).
$$

Suppose $K$ is an infinite field. Then it is a result of Amitsur that the image $ev_d(A_{d,m})$ is a domain. Further its central quotient is a division algebra and is called a universal division algebra. This fact, allows for a proof of the regularity lemma using Gaussian elimination. We refer to \cite{DM-ncrk} for details. 

In particular, Amitsur's results imply this astonishing fact that for every $f$ such that $ev_d(f) \neq 0$, we will be able to find $A_1,\dots,A_m \in \M_{d,d}$ such that $f(A_1,\dots,A_m)$ is invertible. This is no longer true if $K$ is a finite field, and we will exploit this to get counterexamples to the regularity lemma.

%However, when the field is finite, the map $ev$ will not always be injective. Moreover, observe that ${\rm Hom}_{\rm Sets} (\M_{d,d}^m, \M_{d,d})$ is a finite ring! Hence the image of the map $ev$ must be a finite ring. It is easy to observe that $ev(T_1) ev(T_2) \neq ev(T_2) ev(T_1)$ as one can find matrices $A,B \in \M_{d,d}$ that do not commute. This means that the image of $ev$ is necessarily a non-commutative ring as long as $m \geq 2$. Any finite ring that does not contain a zero divisor must be a field, and hence commutative. This means that the image of $ev$ must contain a zero divisor. Thus, there exists an expression $f(T_1,\dots,T_m)$ such that $ev(f)$ is a zero divisor.

%\begin{corollary}
%Assume $K$ is a finite field and $d \in \Z_{\geq 2}$. Then, for $m \geq 2$, there exists a non-commutative polynomial $f(T_1,\dots,T_m)$ such that $ev(f)$ is a zero divisor in $ev(A_{d,m})$.
%\end{corollary}

\subsection{Higman's trick}
Higman's trick is a technique that converts non-commutative polynomials to linear matrices.

\begin{lemma} [Higman's trick] \label{Higman}
Let $f(t_1,\dots,t_m)$ be a non-commutative polynomial. There is a (square) linear matrix $L_f(t_0,\dots,t_m)$ with the property that using elementary (block) row and (block) column operations, for any $(A_1,\dots,A_m) \in \M_{d,d}^m$ for any $d$,  one can transform
$$
L_f(I,A_1,\dots,A_m) \longrightarrow \left[ \begin{array} {cccc|c}
 I & 0 & \dots & 0 & 0 \\
 0 & I &  & \vdots & \vdots \\
  \vdots &  & \ddots & 0 & \vdots \\  
 0 &\dots & 0 & I & 0 \\
 \hline
 0 & \dots &  & 0 & f(A_1,\dots,A_m)\\       
 \end{array} \right].
$$  
\end{lemma}

Higman's trick first appeared in \cite{Higman}. We also refer to \cite{GGOW} for details. We point out that in the above lemma, while $f$ is a polynomial in $m$ indeterminates, $L_f$ is a linear matrix in $m+1$ indeterminates. This extra indeterminate (which is $t_0$) is important. We will not delve into the proof of Higman's trick. We will be content to mention that Higman's trick is constructive. In the cases that we are interested, we will simply exhibit the required $L_f$, and explicitly check that it satisfies the conclusion of Higman's trick.

We will need a modification of the linear matrix obtained from Higman's trick. For $r \in \Z_{\geq 1}$, we consider the linear matrix 
$$
D_{f,r} = \left[\begin{array} {c|ccc} L_f &  & &  \\ \hline & t_0 && \\  &  & \ddots &  \\  &  &  &  t_0 \end{array} \right],
$$
where the lower right block is of size $r \times r$.

\begin{proposition} \label{heavy}
Suppose $f \in A_{d,m}$ such that $f(A_1,\dots,A_m)$ is singular for all $(A_1,\dots,A_m) \in \M_{d,d}^m(K)$, and $ev_d(f) \neq 0$. Then $\rk(D_{f,r}^{\{d,d\}})$ is not a multiple of $d$.
\end{proposition}

\begin{proof}
Suppose the size of the linear matrix $L_f$ is $m \times m$, then $D_{f,r}$ is a linear matrix of size $(m+r) \times (m+r)$. The first step is to prove that $\rk(D_f(A_0,\dots,A_m)) < d(m+r)$ for all $A_0,\dots,A_m \in \M_{d,d}(K)$. If $A_0$ is not invertible, then since $D_{f,r}(A_0,\dots,A_m)$ has diagonal blocks of the form $A_0$, it cannot have full rank. On  the other hand, if $A_0$ is invertible, then 
\begin{align*}
\rk(D_{f,r}(A_0,A_1,\dots,A_m)) &= \rk(L_f(A_0,A_1,\dots,A_m)) + dr \\
&= \rk(L_f(I,A_0^{-1}A_1,\dots,A_0^{-1}A_m)) + dr \\
& = (d(m-1) + \rk(f(A_0^{-1}A_1,\dots,A_0^{-1}A_m)) + dr
\end{align*}

The first equality follows from the structure of $D_{f,r}$, and the second follows from a simple change of basis argument. The last equality follows from Lemma~\ref{Higman}. But now, we know that $\rk(f(A_1,\dots,A_m)) < d$ by hypothesis, so $\rk(D_{f,r}(A_1,\dots,A_m)) < d(m+r)$.

The second step is to show that $rk(D_{f,r}^{\{d,d\}}) > d(m+r -1)$. Observe that since $ev_d(f) \neq 0$, there exist $(B_1,\dots,B_m) \in \M_{d,d}^m(K)$ such that $f(B_1,\dots,B_m) \neq 0$. Thus, by the above computation, we have 
$$
\rk(D_{f,r}(I,B_1,\dots,B_m)) = d(m-1) + \rk(f(B_1,\dots,B_m)) + dr > d(m+ r -1).
$$

Thus we have 
$$
d(m+r-1) < \rk(D_{f,r}^{\{d,d\}}) < d(m+r).
$$

Thus $\rk(D_{f,r}^{\{d,d\}})$ cannot be a multiple of $d$.
\end{proof}

\begin{corollary} \label{light}
Suppose $f$ and $D_{f,r}$ are as in the above proposition. Suppose $\mathcal{X}_{f,r}$ denotes the linear subspace parametrized by the linear matrix $D_{f,r}$. Then $\rk(\mathcal{X}_{f,r}^{\{d,d\}}) = \rk(D_{f,r}^{\{d,d\}})$ is not a multiple of $d$.
\end{corollary}

\section{Proof of main result}
From the previous section, we know that one can generate counterexamples to the regularity lemma by finding non-commutative polynomials $f \in A_{d,m}$ that satisfy the hypothesis of Proposition~\ref{heavy}. So, it remains to find such non-commutative polynomials.

\begin{lemma}
Let $K = F_q$ be the field with $q$ elements, and let $f = T_1^s - T_1 \in A_{d,1}$ where $s = q^d = |K|^d$. Then $f$ satisfies the hypothesis of Proposition~\ref{heavy}.
\end{lemma}

\begin{proof} For any $d \times d$ matrix with entries in $K$, its characteristic polynomial is a polynomial of degree $d$. In particular it's roots lie in some extension of $F_q$ of degree at most $d$. Let $s = q^d$, and let $f = T_1^s - T_1$. Then observe that $ev_d(f) \neq 0$, since $f(E_{1,d}) \neq 0$, where $E_{1,d}$ is the matrix with a $1$ in $(1,d)^{th}$ spot, and $0$ everywhere else. On the other hand, any $A_1 \in \M_{d,d}$ must have an eigenvalue that is a root of $f$. So $f(A_1)$ cannot be invertible for any $A_1 \in \M_{d,d}$.
\end{proof}

\begin{lemma}
Let $f = T_1^s - T_1$. Then the $s \times s$ matrix 
$$
\begin{bmatrix} t_1 & t_0 & && & &\\
  & -t_1 & t_0 & & &&\\
  & & - t_1 & t_0 && &\\
  & & & \ddots & \ddots & &\\
   & & & & -t_1 & t_0 \\
   -t_1 & & & & & -t_1 
 \end{bmatrix}
$$
can be used as the linear matrix $L_f$ in the statement for Higman's trick.
\end{lemma}

\begin{proof}
We have to be a little careful while applying block operations. One can add left multiplied block rows to other block rows, and right multiplied block columns to other block columns. We refer to \cite{DM} for an explanation. 

Consider $L(I,A_1)$, and let us do the following block row and column operations. Left multiply the first block row by $A_1$ and add it to the second block row. Then right multiply the second block column by $A_1$ and subtract from the first block column. The net result is 

$$
\begin{bmatrix} 0 & I & && & &\\
  A_1^2 & 0 & I & & &&\\
  & & - A_1 & I && &\\
  & & & \ddots & \ddots & &\\
   & & & & -A_1 & I \\
   -A_1 & & & & & -A_1 
 \end{bmatrix}.
$$

Now, left multiply the second block row by $A_1$ and add it to the third block row. Then, right multiply the third block column by $A^2$ and subtract it from the first block column. Continuing this process, we end up with 

$$
\begin{bmatrix} 0 & I & && & &\\
0& 0 & I & & &&\\
  \vdots& &  & I && &\\
  \vdots& & & \ddots & \ddots & &\\
   & & & &  & I \\
   A_1^s-A_1 & & & & & 0
 \end{bmatrix}.
$$

Then by permuting (block) columns, we can get it to the required form.

\end{proof}

\begin{proof} [Proof of Theorem~\ref{main}]
Taking $f = T_1^s - T_1$ where $s = |K|^d$. If we take $L_f$ as in the above lemma, then for $r = n - |K|^d \geq 1$, we have $D_{f,r} = D (t_0,t_1)$ and $\mathcal{X}_{f,r} = \mathcal{X}$, the linear subspace in the statement of the theorem. The theorem follows from Corollary~\ref{light}.
\end{proof}

\begin{remark}
We can give another interesting example when $K = F_2$. We leave it to the reader to check that $f(T_1,T_2) = [T_1,T_2]^2 + [T_1,T_2] \in A_{2,2}$ satisfies the hypothesis of Proposition~\ref{heavy}. Using the same ideas, we construct another counterexample for the regularity lemma. Consider the linear subspace $$
\mathcal{X} = \left\{ \left[ \begin{array}{ccc|ccc|c} 
0 & d & a &  0 & d & a & 0 \\
-d & 0 & b &  -d & 0 & b  & 0\\
a & b & 0  & a & b & 0 & 0\\
\hline
0 & d & a  & 0 & 0 & 0 & 0\\
-d & 0 & b  & 0 & 0 &0  & 0 \\
a & b & 0 & 0 & 0 & d & 0\\
\hline
0 & 0 & 0 & 0 & 0 & 0 & d
 \end{array} \right] :  a,b,c,d \in K \right\}.
$$
Then $\rk(\mathcal{X}^{\{2,2\}})$ is not a multiple of $2$.
\end{remark}

\end{document}